\documentclass[12pt]{amsart}
\usepackage{amssymb}
\newtheorem{theorem}{Theorem}
\newtheorem{lemma}[theorem]{Lemma}

\newtheorem{question}{Question}

\title{A regular non-weakly discretely generated $P$-space}

\author{Santi Spadaro}

\address{Dipartimento di Ingegneria\\
Universit\`a degli Studi di Palermo\\
viale delle Scienze 8\\
90128 Palermo, Itay}

\email{santidspadaro@gmail.com}
\email{santidomenico.spadaro@unipa.it}

\author{Paul Szeptycki}

\address{Department of Mathematics\\
Faculty of Science and Engineering\\
York University\\
Toronto, ON M3J 1P3, Canada}

\email{szeptyck@yorku.ca}



\begin{document}

\begin{abstract} 
We construct a consistent example of a topological space $Y=X \cup \{\infty\}$ such that:

\begin{enumerate}
\item $Y$ is regular.
\item Every $G_\delta$ subset of $Y$ is open.
\item The point $\infty$ is not isolated, but it is not in the closure of any discrete subset of $X$.
\end{enumerate}
 \end{abstract}

\maketitle

\section{Introduction}
A space $X$ is discretely generated if its topology is determined by its discrete subsets, that is, if for every subset $A$ of $X$ and every point $x \in \overline{A}$ there is a discrete set $D \subset A$ such that $x \in \overline{D}$. Discretely generated spaces were introduced by Dow, Tkachenko, Tkachuk and Wilson in \cite{DTTW}. 

Discrete generability may be considered a convergence-type property. Indeed, Fr\'echet-Urysohn spaces and radial spaces are discretely generated. Actually, the authors of \cite{DTTW} proved that even sequential spaces and compact spaces of countable tightness are discretely generated. However, Bella and Simon \cite{BS} constructed, under the continuum hypothesis, a compact pseudoradial space which is not discretely generated. A similar example was constructed by Aurichi \cite{Au} assuming the existence of a Suslin Tree.

Somewhat surprisingly, monotonically normal spaces are discretely generated \cite{DTTW}. So, in particular, every subspace of a linearly ordered topological space is discretely generated. 

One may also consider the following natural weakening of the notion of discrete generability: a space $X$ is \emph{weakly discretely generated} if for every non closed set $A \subset X$ there is a discrete set $D \subset A$ such that $\overline{D} \setminus A \neq \emptyset$. Using the well-known fact that a space is compact if and only if the closure of each one of its discrete subsets is compact (see \cite{Tk}), one readily sees that every compact space is weakly discretely generated.

Bella and Simon \cite{BS} were the first to study discrete generability of the class of $P$-spaces, that is spaces whose $G_\delta$ subsets are open. A $P$-space of character $\omega_1$ is discretely generated, because given a non-closed set $A \subset X$ and a point $x \in \overline{A}$ there is an $\omega_1$-sequence inside $A$ converging to $x$ and every $\omega_1$-convergent sequence without the limit point is a discrete set in a $P$-space. Bella and Simon proved that if $X$ is a regular $P$-space of countable extent and the tightness of $X$ does not exceed $\omega_1$, then $X$ is discretely generated. Moreover, under the principle $P_1$, if $X$ is a $P$-space of tightness $\omega_1$ and character less than $2^{\omega_1}$, then $X$ is discretely generated. It is not known whether every Lindel\"of $P$-space must be weakly discretely generated. Given that Lindel\"of $P$-spaces have a very similar behavior to that of compact spaces, one would expect this question to have a positive answer, but only partial answers have been provided to it so far. For example, Alas, Junqueira and Wilson \cite{AJW} proved that every Lindel\"of $P$-space which has either pseudocharacter less than $\aleph_\omega$ or character at most $\aleph_\omega$ is weakly discretely generated. They also proved that if $X$ is a Lindel\"of $P$-space where every linearly Lindel\"of subspace is Lindel\"of then $X$ is weakly discretely generated. More results about discrete generability of $P$-spaces appear in the article \cite{AJW2} by the same authors.

Both the authors of \cite{BS} and the authors of \cite{AJW} provided various examples of $P$-spaces which are not discretely generated or even not weakly discretely generated. Unfortunately, they are all Hausdorff non-regular spaces. 

It is the aim of this note to construct the first consistent example of a regular $P$-space which is not weakly discretely generated.

All spaces are assumed to be regular. Given a space $X$ we denote by $X_\delta$ the \emph{$G_\delta$}-modification of $X$, that is, the space whose underlying set is $X$ and whose topology is generated by the $G_\delta$ subsets of $X$. Undefined notions may be found in \cite{En}, \cite{Je} and \cite{JCard}.

\section{The main result}

Our example is a modification on the construction of an HFC (see \cite{J}).

\begin{theorem} \label{main}
Assume $2^{\aleph_0}=\aleph_1$ and $2^{\aleph_1}=\aleph_2$. Then there is a regular non-weakly discretely generated $P$-space.
\end{theorem}

Before constructing our space, let us recall the following higher cardinal variant of the $\Delta$-system Lemma (see, for example, Theorem 9.19 of \cite{Je}).

\begin{lemma} \label{Deltasystem}
Let $\kappa$ and $\lambda$ be cardinals such that $\kappa^\lambda=\kappa$. Let $\mathcal{A}$ be a collection of sets of cardinality $\lambda$ such that $|\mathcal{A}|=\kappa^+$. Then there is a subfamily $\mathcal{D} \subset \mathcal{A}$ such that $|\mathcal{D}|=\kappa^+$ and a set $R$ such that $A \cap B=R$, for every distinct $A,B \in \mathcal{D}$.
\end{lemma}

\begin{proof} [Proof of Theorem $\ref{main}$]
Let $\{U_\alpha: \alpha < \aleph_2\}$ be an enumeration of all open subsets $U \subset (2^{\omega_2})_\delta$ such that $U=\bigcup \{[\sigma_\alpha]: \alpha < \omega_1\}$ and $\{dom(\sigma_\alpha): \alpha < \omega_1 \}$ is a pairwise disjoint family of countable sets. Let $\{C_\alpha: \alpha < \omega_2\}$ be an enumeration of all countable subsets of $\omega_2$ such that, for every $\alpha < \omega_2$, $C_\alpha \subset \alpha$ and for every $C \subset \omega_2$, the set $\{\alpha < \omega_2: C_\alpha = C \}$ has cardinality $\omega_2$.

We will recursively choose an increasing sequence of ordinals $\{\gamma_\alpha: \alpha < \omega_2 \}$ and points $x_\alpha \in 2^{\omega_2}$ in such a way that, if we let $W_\alpha=\{x \in 2^{\omega_2}: x(\gamma_\alpha)=1\}$, then:

$$(*)_\alpha \quad x_\alpha \in \bigcap \{U_\beta: \beta \leq \alpha\} \cap \bigcap \{W_\beta: \beta \in C_\alpha \}$$

and such that the $W_\alpha$'s left separate $X=\{x_\alpha: \alpha < \omega_2\}$.

Suppose that $\alpha < \omega_2$ is given and that $\gamma_\beta$ and $x_\beta$ have been defined for every $\beta < \alpha$ in such a way that:

\begin{enumerate}
\item For each $\beta < \alpha$, $x_\beta (\gamma_\beta)=1$ and $x_\beta(\gamma)=0$ for every $\gamma > \gamma_\beta$.
\item For each $\beta < \alpha$, $x_\beta$ satisfies $(*)_\beta$
\item For each $\beta < \eta < \alpha$, $\gamma_\beta < \gamma_\eta$.
\end{enumerate}

To define $x_\alpha$ and $\gamma_\alpha$, first re-enumerate $\{U_\beta: \beta < \alpha \}$ as $\{V_\xi: \xi < \omega_1\}$. Recall that each $V_\xi$ can be represented as $\bigcup \{[\sigma_{\xi, \eta}]: \eta<\omega_1\}$, where $\{dom(\sigma_{\xi, \eta}): \eta < \omega_1 \}$ is a pairwise disjoint family of countable subsets of $\omega_2$. Therefore we can choose $\eta_\xi$, for each $\xi < \omega_2$ so that the elements of the set $\{dom(\sigma_{\xi, \eta_\xi}): \xi < \omega_1 \}$ are pairwise disjoint and each of them is disjoint from $\{\gamma_\beta: \beta \in C_\alpha \}$.

Next choose $\gamma_\alpha$ larger than all the previously chosen $\gamma_\beta$ and also above $\bigcup \{dom(\sigma_{\xi, \eta_\xi}): \xi < \omega_1 \}$. Since all the domains are pairwise disjoint we can now choose:

$$x_\alpha \in \bigcap \{W_\beta: \beta \in C_\alpha\} \cap \bigcap \{[\sigma_{\xi, \eta_\xi}]: \xi < \omega_1 \}$$

We can also require that $x_\alpha (\gamma_\alpha)=1$ and $x_\alpha(\gamma)=0$, for every $\gamma > \gamma_\alpha$, since $\gamma_\alpha$ is above all the domains of the open sets in the above intersection.

At the end of the induction let $X=\{x_\alpha: \alpha < \omega_2\}$ equipped with the topology inherited from $(2^{\omega_2})_\delta$. Obviously $X$ is a regular $P$-space.

\bigskip

\noindent {\bf Claim.} $hL(X)=\aleph_1 < \aleph_2= d(X)=w(X)=|X|$.

\begin{proof}[Proof of Claim 2] Notice that the $W_\alpha$'s left separate $X$ and hence $d(X)=\omega_2$.

It remains to prove that $hL(X)=\aleph_1$. Observe, first of all that, by construction, $|X \setminus U_\alpha| \leq \omega_1$, for every $\alpha < \omega_2$. 

Suppose by contradiction that $hL(X) > \omega_1$. Then we can find an $\omega_2$-sized right separated subspace $\{y_\alpha: \alpha < \omega_2 \}$ of $X$. That means we can find countable partial functions $\{\sigma_\alpha: \alpha < \omega_2 \}$ such that $y_\alpha \in [\sigma_\alpha]$ and $y_\beta \notin [\sigma_\alpha]$, for every $\beta > \alpha$. Since $\omega_2^\omega=\omega_2$, we can apply Lemma $\ref{Deltasystem}$ to find an $\omega_2$-sized set $S \subset \omega_2$ such that the family $\{dom(\sigma_\alpha): \alpha \in S \}$ forms a $\Delta$-system with root $R$. Let $\sigma'_\alpha=\sigma_\alpha \upharpoonright (dom(\sigma_\alpha) \setminus R)$.  Since $|2^R|=2^{\aleph_0}=\aleph_1$, then there has to be $T \in [S]^{\omega_2}$ such that $\sigma_\alpha \upharpoonright R=\sigma_\beta \upharpoonright R$, for every $\alpha, \beta \in T$. Hence, for every $\alpha \in T$ and for every $\beta \in T$ such that $\beta > \alpha$ we have $y_\beta \notin [\sigma'_\alpha]$.

Now let $U=\bigcup \{[\sigma'_\alpha]: \alpha < \omega_1 \}$. Since $\{dom(\sigma'_\alpha): \alpha < \omega_1 \}$ is an $\omega_1$-sized pairwise disjoint family of countable sets, there is $\delta < \omega_2$ such that $U=U_\delta$.

However $|X \setminus U| > \aleph_1$, which is a contradiction.
\renewcommand{\qedsymbol}{$\triangle$}
\end{proof}

Choose now a point $\infty \notin X$ and let $Y=X \cup \{\infty\}$. Points of $X$ have their usual neighbourhoods, while a basic open neighbourhood of $\infty$ is a set of the form.

$$W_C=\{\infty\} \cup \bigcap \{W_\alpha: \alpha \in C \}$$

where $C$ is a countable subset of $\omega_2$.

Clearly $Y$ is a $P$-space and since, for every countable $C$, the set $W_C \cap X$ is clopen in $X$, the topology of $Y$ is still regular. Since discrete sets in $X$ have cardinality bounded by $\omega_1$ and the $W_\alpha$'s left separate $X$, it follows that $\infty$ is not in the closure of any discrete subset of $X$, but it is in the closure of $X$, because $W_C$ is non-empty for every countable set $C$. This implies that $Y$ is not weakly discretely generated.

\end{proof}

\section{Open Questions}

The following natural question is still open.

\begin{question}
Is there a ZFC example of a regular non-weakly discretely generated $P$-space?
\end{question}

Our example is not Lindel\"of, so the following question of Bella and Simon from \cite{BS} is also open.

\begin{question}
Is there a regular Lindel\"of non-discretely generated $P$-space?
\end{question}

\section{Acknowledgements}

The first author is grateful to INdAM-GNSAGA for partial financial support.

\end{document}